\documentclass[preprint,12pt]{elsarticle}

\usepackage{amssymb}
\usepackage{amsmath}
\usepackage{amsthm}
\usepackage{graphicx}
\usepackage{setspace}
\usepackage{algorithm}
\usepackage{algpseudocode}
\usepackage{subcaption}  
\usepackage{booktabs}    
\usepackage{multirow}    
\usepackage{natbib}
\usepackage{tabularx}  
\usepackage{url}

\newtheorem{lemma}{Lemma}
\newtheorem{theorem}{Theorem}
\newtheorem{proposition}{Proposition}
\newtheorem{remark}{Remark}

\newcommand{\prob}{\mathbb{P}}

\newcommand{\norm}[1]{\left\lVert#1\right\rVert}

\begin{document}

\begin{frontmatter}

\title{Statistical Inference for Cumulative INAR($\infty$) Processes via Least-Squares}

\author{Ying-Li Wang\corref{cor1}}
\ead{2022310119@163.sufe.edu.cn}

\author{Xiao-Hong Duan}
\ead{isduanxh@163.com}

\author{Ping He}
\ead{pinghe@mail.shufe.edu.cn}

\cortext[cor1]{Corresponding author}

\address{School of Mathematics, Shanghai University of Finance and Economics, Shanghai, 200433, China}

\begin{abstract}
This paper investigates the cumulative Integer-Valued Autoregressive model of infinite order, denoted as INAR($\infty$), a class of processes crucial for modeling count time series and equivalent to discrete-time Hawkes processes. We propose a computationally efficient conditional least-squares (CLS) estimator to address the challenge of parameter inference in this infinite-dimensional setting. We establish the key theoretical properties of the estimator, including its consistency and asymptotic normality. A central contribution is the rigorous treatment of its large-sample distribution in a framework where the parameter dimension grows with the sample size, for which we derive the corresponding sandwich-form covariance matrix. The theoretical results are substantiated through comprehensive Monte Carlo simulations. These experiments demonstrate that the estimator's accuracy and stability systematically improve as the sample size increases, confirming its consistency. Furthermore, we show that the estimator's finite-sample distribution is well-approximated by a normal distribution, and this approximation becomes more robust with larger samples. Our work provides a complete and practical framework for statistical inference in cumulative INAR($\infty$) models. The code to reproduce the numerical experiments is publicly available at \url{https://github.com/gagawjbytw/INAR_estimation}.
\end{abstract}


\end{frontmatter}

\section{Introduction}
The INAR($\infty$) process is an integer-valued time series model that extends the traditional INAR($p$) processes to infinite order (see, for example, \citet{kirchner2016hawkes}). For $\alpha_k \ge 0$, where $k$ is a non-negative integer, let $(\epsilon_n)_{n\in\mathbb Z}$ be i.i.d. Poisson($\nu$) random variables, and let $\xi_l^{(n,k)}$ be Poisson$(\alpha_k)$ random variables. These variables are mutually independent for different $n \in \mathbb{Z}$, $k \in \mathbb{N}$, and $l \in \mathbb{N}$, and they are also independent of the sequence $(\epsilon_n)_{n\in\mathbb Z}$.

An INAR($\infty$) process is a sequence of random variables $(X_n)_{n \in \mathbb{Z}}$ that satisfies the following system of stochastic difference equations:
\begin{align*}
\epsilon_n &= X_n - \sum_{k=1}^\infty \alpha_k \circ X_{n-k}= X_n - \sum_{k=1}^\infty \sum_{l=1}^{X_{n-k}} \xi_l^{(n,k)}, \quad n \in \mathbb{Z},
\end{align*}
where the operator $``\circ"$, called the \textbf{reproduction operator}, is defined as $\alpha \circ Y := \sum_{n=1}^Y \xi_n^{(\alpha)}$, for a random variable $Y$ that takes non-negative integer values and a constant $\alpha \ge 0$. Here, $\left(\xi_n^{(\alpha)}\right)_{n\in\mathbb N}$ are i.i.d. Poisson$(\alpha)$ random variables and are independent of $Y$. We refer to $\xi_n^{(\alpha)}$ as the \textbf{offspring variable}, and to $\left(\xi_n^{(\alpha)}\right)$ as the \textbf{offspring sequence}. Additionally, we call $\nu$ the \textbf{immigration parameter}, $(\epsilon_n)$ the \textbf{immigration sequence}, and $\alpha_k \ge 0$ the \textbf{reproduction coefficient} for each non-negative integer $k$.

A cumulative INAR($\infty$) process, also known as a discrete Hawkes process, is defined by $N_n = \sum_{s=1}^n X_s$. Hawkes processes, introduced by \citet{hawkes1971spectra}, are continuous-time self-exciting point processes widely used in various fields. A general Hawkes process is a simple point process $N$ admitting an $\mathcal F_t^{-\infty}$ intensity 
\[
  \lambda_t:=\lambda\left( \int_{-\infty}^th(t-s)N(ds) \right),
\]
where $\lambda(\cdot):\mathbb R^+\rightarrow\mathbb R^+$ is locally integrable and left continuous, $h(\cdot):\mathbb R^+\rightarrow\mathbb R^+$, and we always assume that $\norm h_{L^1}=\int_0^\infty h(t)dt<\infty$. We always assume that $N(-\infty,0]=0$, i.e. the Hawkes process has empty history. In the literature, $h(\cdot)$ and $\lambda(\cdot)$ are usually referred to as the exciting function and the rate function, respectively. The Hawkes process is linear if $\lambda(\cdot)$ is linear and it is nonlinear otherwise, in the linear case, the stochastic intensity can be written as 
\[
  \lambda_t=\nu+\int_0^{t-}h(t-s)N(ds).
\]

Discrete-time analogs, such as cumulative INAR($\infty$) processes, offer similar modeling capabilities with a focus on count data observed at fixed time intervals. Under certain conditions, the Poisson autoregressive process can be viewed as an INAR($\infty$) process with Poisson offspring. For a comprehensive discussion of Poisson autoregressive models and their connections to INAR and Hawkes processes, refer to \citet{fokianos2021multivariate} and \citet{huang2023nonlinear}. It is easy to see that if we let an INAR($\infty$) process $(X_n)_{n \geq 1}$ start from time 1 ($X_1\sim \text{Poisson}(\nu)$), it can also be defined by:
\begin{equation}\label{process_definition}
\lambda_n = \nu + \sum_{s=1}^{n-1} \alpha_{n-s} X_s,
\end{equation}
where $\nu > 0$ is the immigration rate, and $(\alpha_n)_{n \geq 1} \in \ell^1$ represents the offspring distribution, with $\alpha_n \geq 0$ for all $n \in \mathbb{N}$. Given the history $\mathcal{F}_{n-1}$, the count $X_n$ follows a Poisson distribution with parameter $\lambda_n$, i.e.,
\[
X_n \mid \mathcal{F}_{n-1} \sim \text{Poisson}(\lambda_n).
\]
INAR($\infty$) processes are very powerful tools for estimating Hawkes processes; see for example, \citet{kirchner2017estimation}. The INAR($\infty$) process is in fact a series of discretized time observations of a continuous-time linear Hawkes process, where the exciting function is
\begin{equation}\label{generalizedexpressionofdiscrete}
h(t) = \sum_{k=1}^\infty \alpha_k \delta_{\{t=k\}},
\end{equation}
where $\delta$ is the generalized Delta function. This can be understood from the immigration-birth representation of the continuous-time Hawkes process. Consider the population of a region: if an immigrant arrives at time $t$ (either as a descendant of a former immigrant or from another region), the number of descendants of the immigrant at time $t+n$ follows a Poisson distribution with parameter $\alpha_n$. Denote $X_n$ as the increase in population volume in the time interval $(n-1, n]$; then it consists of two parts:
\begin{enumerate}
    \item The first part is the number of new immigrants from other regions, which follows a Poisson distribution with parameter $\nu$.
    \item The second part is the number of descendants from before time $n$, which follows a Poisson distribution with parameter $\sum_{k=1}^{n-1} \alpha_k X_{n-k}$.
\end{enumerate}
As a result, $X_n \mid \mathcal{F}_{n-1} \sim \text{Poisson}(\nu + \sum_{k=1}^{n-1} \alpha_k X_{n-k})$.

In this paper, we develop a comprehensive framework for the estimation and inference of the cumulative INAR($\infty$) process. Our primary contribution is the introduction of a computationally straightforward conditional least-squares (CLS) estimator, which circumvents the complexities often associated with likelihood-based methods for infinite-order models. The novelty of our approach lies in three key areas:

\begin{enumerate}
    \item We first establish a new estimation framework by defining a least-squares contrast function and showing its equivalence to a valid distance metric in the parameter space.
    
    \item We rigorously derive the large-sample properties of our CLS estimator, proving its consistency. Critically, we address the theoretical challenge of an estimator whose dimension grows with the sample size $T$.
    
    \item We establish the \emph{approximate normality} of the estimator for large but finite $T$, providing a practical and theoretically sound basis for statistical inference, such as constructing confidence intervals and hypothesis tests.
\end{enumerate}

Our work thus provides a complete and accessible methodology for analyzing discrete-time, self-exciting count data, supported by both rigorous proofs and extensive numerical validation.

\section{Main Results}
The technical method in this paper is inspired by \citet{reynaud2010adaptive}. Let us give some notations first. In this paper, \( \|\cdot\|_1 \) and \( \|\cdot\|_2 \) denote the usual \( \ell^1 \)-norm and \( \ell^2 \)-norm, respectively. 
We also set $(A_n)_{n\ge1}\in \ell^1$ as the sequence defined on $\mathbb N$ by $A_n=\sum_{k=1}^\infty(\alpha)_n^{*k}$, where $*$ denotes the discrete convolution which means for two non-negative sequences $(q_n)_{n\ge1}$, $(m_n)_{n\ge1}\in \ell^1$, $(q*m)(n)=\sum_{s=1}^{n-1}q_sm_{n-s}$, and $\alpha^{*(k+1)}$ denotes the discrete convolution of $\alpha^{*k}$ with $\alpha$, i.e., $\alpha^{*(k+1)}=\alpha*\alpha^{*k}$. $(A_n)_{n\ge1}$ is well defined since $\norm {\alpha}_1<1$.

\subsection{Problem Formulation}

The parameter we aim to estimate is $s = (\nu, \alpha)$, where \( \alpha = (\alpha_1, \alpha_2, \cdots) \). Since observational data are always finite, we introduce a sufficiently large integer \( T \) (with \( T \) increasing as the data length increases). Then, we estimate $s = (\nu, \alpha_1, \alpha_2, \cdots, \alpha_{T-1})$.
We assume $\sum_{k=1}^{T-1} \alpha_k < 1$ to ensure the stationarity of the process.

The parameter space is a Euclidean space 
\[
  \mathfrak{l}^2 = \{f : f = (\mu, \beta) = (\mu, \beta_1, \beta_2, \cdots, \beta_{T-1})\}
\]
equipped with the inner product \( \langle \cdot, \cdot \rangle \), where for \( f = (\mu, \beta) \) and \( g = (\xi, \gamma) \) in \( \mathfrak{l}^2 \),
$\langle f, g \rangle = \mu \xi + \sum_{k=1}^{T-1} \beta_k \gamma_k.$

\subsection{Least-Squares Contrast}

For \( f = (\mu, \beta) \in \mathfrak{l}^2 \), we define the intensity candidates as
\[
\Phi_f(n) := \mu + \sum_{k=1}^{n-1} \beta_k X_{n-k},
\]
and, in particular, \( \Phi_s(n) = \lambda_n \). We want to estimate the intensity \( \Phi_s(n) \). The estimator \( \Phi_f(n) \) should be sufficiently close to \( \Phi_s(n) \). For every \( f \in \mathfrak{l}^2 \), we define a Least-Squares Contrast:
\[
\gamma_T(f) := -\frac{2}{T} \sum_{n=1}^T \Phi_f(n) X_n + \frac{1}{T} \sum_{n=1}^T \Phi_f^2(n).
\]
Now, let's prove that \( \gamma_T(f) \) can be used as a metric to measure the distance between \( \Phi_f(n) \) and \( \Phi_s(n) \). First, for every \( f \in \mathfrak{l}^2 \), we define
\[
  D_T^2(f):=\frac{1}{T}\sum_{n=1}^T \Phi_f^2(n)\ \text{and}\ \|f\|_D := \sqrt{\mathbb{E}[D_T^2(f)]}.
\]
Proposition \ref{prop:quadraticform} guarantees that \( D_T^2 \) is a quadratic form and that \( \|f\|_D \) is equivalent to \( \|f\|_2 \). To prove Proposition \ref{prop:quadraticform}, we first introduce some technical lemmas.

\begin{lemma}[Solution of Discrete Renewal Equations]\label{discreterenewalequation}
  Given a non-negative sequence $(\alpha_n)_{n\ge1} \in \ell^1$ and two non-negative sequences $(x_n)_{n\ge1}$, $(y_n)_{n\ge1}$, the following equation
  \begin{equation}\label{seriesconvolution}
    x_n = y_n + \sum_{s=1}^{n-1} \alpha_s x_{n-s}
  \end{equation}
  has the unique solution $x_n = \left(y + y * A\right)(n) = y_n + \sum_{i=1}^{n-1} A_i y_{n-i}$.
\end{lemma}

\begin{proof}
  We provide the proof in \ref{sec:appendix_proof_discrete_renewal_equation}.
\end{proof}

 From Lemma \ref{discreterenewalequation}, we can easily obtain an upper bound for $\mathbb E[\lambda_n]$. In fact, taking the expectation on both sides of \eqref{process_definition}, we have $\mathbb E[X_n] = \nu + \sum_{s=1}^{n-1} \alpha_{n-s} \mathbb E[X_s].$
Using Lemma \ref{discreterenewalequation}, it follows that
\begin{equation}\label{upperboundlambda}
  \mathbb E[\lambda_n]=\mathbb E[X_n]\le \frac{\nu}{1-\norm \alpha_1}.
\end{equation}
An upper bound of $\mathbb E[X_n^2]$ is obtained when $\norm \alpha_2^2<\frac12$,
\begin{align*}
   \mathbb E[X_n^2]-\mathbb E[\lambda_n]
   =\mathbb E[\lambda_n^2]
   =\mathbb E\left[ \left( \nu+\sum_{k=1}^{n-1}\alpha_kX_{n-k} \right)^2 \right]\le 2\mathbb E\left[ \nu^2+\sum_{k=1}^{n-1}\alpha_k^2X_{n-k}^2 \right].
\end{align*}
Therefore, 
\[
  \mathbb E[X_n^2]\le \frac{2\nu^2+\mathbb E[\lambda_n]}{1-2\norm \alpha_2^2}\le \frac{2\nu^2(1-\norm \alpha_1)+\nu}{(1-2\norm\alpha_2^2)(1-\norm\alpha_1)}.
\]

\begin{remark}
  We believe that $\|\alpha\|_2^2 < \frac{1}{2}$ appears to be a technical requirement for deriving the upper bound. In the numerical experiments, we also set $\alpha_1=0.8$ and $\alpha_n=0$ for $n \geq 2$. Our results show that the relative error falls within an acceptable range, as defined in our analysis.
\end{remark}

\begin{lemma}\label{boundlemma}
  Let \( (N_n)_{n\ge1} \) be a cumulative INAR$(\infty)$ process with $\norm \alpha_2^2<\frac12$, and \( \beta = (\beta_1, \beta_2, \dots) \in \ell^1 \) with \( \beta_k \geq 0 \) for \( k \geq 1 \). Then, for every \( n\in\mathbb N \),
  \[
    \mathbb{E}\left[ \left( \sum_{k=1}^{n-1} \beta_k X_{n-k} \right)^2 \right] \leq \frac{2\nu^2(1-\norm \alpha_1)+\nu}{(1-2\norm\alpha_2^2)(1-\norm\alpha_1)} \left( \sum_{k=1}^{n-1} \beta_k \right)^2.
  \]
\end{lemma}

\begin{proof}
  We provide the proof in \ref{sec:appendix_proof_boundlemma}.
\end{proof}

\begin{proposition}\label{prop:quadraticform}
  \( D_T^2 \) is a quadratic form on \( \mathfrak{l}^2 \). Assume $\norm \alpha_2^2<\frac12$, the squared expectation of \( D_T^2 \) is \( \|\cdot\|_D^2 \), and it satisfies the following inequality:
  \begin{equation}\label{equivalentnorm}
    L\|f\|_2 \leq \|f\|_D \leq K\|f\|_2,
  \end{equation}
  where 
  \[
    L^2 = \min \left\lbrace   \dfrac{1}{1+\nu T(T-1)(1+\| \alpha \|_1)^2},\dfrac{\nu  }{2 T (1-\| \alpha \|_1)(1+\| \alpha \|_1)^2}\right\rbrace,
  \]
  and 
  \[
    K^2 = \max\left\{ 2, \frac{T-1}{2}\left[ \frac{2\nu^2}{(1-\norm \alpha_1)^2}+\frac{2\nu^2(1-\norm\alpha_1)+\nu}{(1-2\norm \alpha_2^2)(1-\norm\alpha_1)} \right]\right\}.
  \]
\end{proposition}

\begin{proof}
  We provide the proof in \ref{sec:appendix_proof_quadraticform}.
\end{proof}

Then we can give our main theorem.

\begin{theorem}\label{thm:maintheorem}
  Let $(N_n)_{n\ge1}$ be a cumulative INAR($\infty$) process with $\norm \alpha_1<1$ and $\norm \alpha_2^2<\frac12$, for any $f\in\mathfrak{l}^2$, define
  \[
    \gamma_T(f) :=-\frac2T\sum_{n=1}^T\Phi_f(n)X_n+\frac1T\sum_{n=1}^T\Phi_f^2(n),
  \]
  then $\gamma_T(f)$ is a contrast, i.e. $\mathbb E[\gamma_T(f)]$ reaches its minimum when $f=s$.
\end{theorem}

\begin{proof}
  We provide the proof in \ref{sec:appendix_proof_maintheorem}.
\end{proof}

Finally, we will give the exact expression of $\gamma_T(f)$ as follows,
\begin{equation*}
	\begin{aligned}
	\gamma_T(f) = & -\frac{2}{T}\sum_{n=1}^{T}\Phi_f(n)X_n + \frac{1}{T}\sum_{n=1}^{T}\Phi^2_f(n) \\
	= & -\frac{2}{T}\sum_{n=1}^{T}\left( \mu + \sum_{k=1}^{n-1}\beta_kX_{n-k}\right) X_n+
	\frac{1}{T}\sum_{n=1}^{T}\left( \mu + \sum_{k=1}^{n-1}\beta_kX_{n-k} \right) ^2\\
	= & -2\left[  \left( \frac{1}{T}\sum_{n=1}^{T}X_n\right) \mu + \sum_{k=1}^{T-1}\left( \frac{1}{T} \sum_{n=k+1}^{T}X_{n-k}X_n\right) \beta_k \right]  \\
	& + \mu ^2 + \sum_{k=1}^{T-1} \beta_k^2 \left(\frac{1}{T}\sum_{n=k+1}^{T}X_{n-k}^2 \right) +
	2\sum_{k=1}^{T-1}\mu \beta_k\left(\frac{1}{T}\sum_{n=k+1}^{T}X_{n-k} \right) \\
	& +2\sum_{i=1}^{T-1}\sum_{j=i+1}^{T-1}\beta_i\beta_j\left( \frac{1}{T}\sum_{n=j+1}^{T}X_{n-i}X_{n-j}\right) 
	.
	\end{aligned}
\end{equation*}
Assume $\boldsymbol{\theta}_T$ to be the $T$-dimensional vector consisting of the parameters to be estimated,
\[
\boldsymbol{\theta}_T = \left( \mu, \beta_1, \cdots, \beta_k, \cdots, \beta_{T-1} \right)^\top.
\]
  then we can rewrite $ \gamma_T(f)  $ into the following form:
	\begin{equation*}
	\gamma_T(f) = -2 \boldsymbol{\theta}_T^\top \boldsymbol{b} + \boldsymbol{\theta}_T^{\top} \boldsymbol{Y} \boldsymbol{\theta}_T,
	\end{equation*}
	where
	\[
\boldsymbol{b} = \left( 
\frac{1}{T}N_T, \; 
\frac{1}{T}\sum_{n=2}^{T}X_{n-1}X_n, \; 
\cdots, \; 
\frac{1}{T}\sum_{n=k+1}^{T}X_{n-k}X_n, \; 
\cdots, \; 
\frac{1}{T}\sum_{n=T}^{T}X_1X_n 
\right)^\top,
\]
  and
	\begin{equation*}
  \resizebox{\columnwidth}{!}{$
	\boldsymbol{Y} = \left(
	\begin{array}{cccccc}
	1   & \frac{1}{T}\sum_{n=2}^{T}X_{n-1} &\cdots & \frac{1}{T}\sum_{n=k+1}^{T}X_{n-k} & \cdots & \frac{1}{T}X_1\\
	\frac{1}{T}\sum_{n=2}^{T}X_{n-1} & \frac{1}{T}\sum_{n=2}^{T}X_{n-1}^2 & \cdots & \frac{1}{T}\sum_{n=k+1}^{T}X_{n-k}X_{n-1} & \cdots & \frac{1}{T}X_1X_{T-1}\\
  \frac{1}{T}\sum_{n=3}^{T}X_{n-2} & \frac{1}{T}\sum_{n=3}^{T}X_{n-1}X_{n-2} & \cdots & \frac{1}{T}\sum_{n=k+1}^{T}X_{n-k}X_{n-2} & \cdots & \frac{1}{T}X_1X_{T-2}\\
	\vdots & \vdots & \ddots & \vdots & \vdots & \vdots \\
	\frac{1}{T}\sum_{n=k+1}^{T}X_{n-k} & \frac{1}{T}\sum_{n=k+1}^{T}X_{n-1}X_{n-k} & \cdots & \frac{1}{T}\sum_{n=k+1}^{T}X_{n-k}^2  & \cdots & \frac{1}{T}X_1X_{T-k}\\
  \frac{1}{T}\sum_{n=k+2}^{T}X_{n-k-1} & \frac{1}{T}\sum_{n=k+2}^{T}X_{n-1}X_{n-k-1} & \cdots & \frac{1}{T}\sum_{n=k+2}^{T}X_{n-k}X_{n-k-1}  & \cdots & \frac{1}{T}X_1X_{T-k-1}\\
	\vdots & \vdots & \vdots & \vdots & \ddots & \vdots \\
	\frac{1}{T}X_1& \frac{1}{T}X_{T-1}X_1 & \cdots & \frac{1}{T}X_{T-k}X_{1} & \cdots & \frac{1}{T}X_1^2
	\end{array}
	\right),$}
	\end{equation*}
  precisely,
  \begin{equation*}
     \boldsymbol{Y}=(Y_{ij})
     =
     \begin{cases}
     1, &i=j=1,\\
     \frac{1}{T}\sum_{n=\max\{i,j\}}^TX_{n-\max\{i,j\}+1}, &i\neq j,\ i\ \text{or}\ j=1,\\
     \frac1T\sum_{n=\max\{i,j\}}^TX_{n-i+1}X_{n-j+1}, &\text{otherwise}.
     \end{cases}
  \end{equation*}

    The Conditional Least-Squares (CLS) method is a well-established approach for estimating the parameters of INAR processes, as it often circumvents the numerical complexities associated with Maximum Likelihood Estimation. The CLS estimation for the univariate INAR($p$) model has been discussed by \citet{DuLi1991} and \citet{Zhang2010}, frequently building upon the general theoretical framework for CLS estimators developed by \citet{KlimkoNelson1978}. A key insight, noted by \citet{Latour1997} for the multivariate case, is that an INAR($p$) process can be represented as a standard vector autoregressive (VAR) process with white-noise innovations. This representation, detailed in texts such as \citet{Lutkepohl2005}, allows for a straightforward derivation of the CLS estimator. Following this principle, we can solve for the parameter vector $\theta$.

  By using the conclusion of the general least squares method, the $ \hat{\boldsymbol{\theta}}_T $ that minimizes $ \gamma_T(f) $ satisfies
	$ 
	\boldsymbol{Y}\hat{\boldsymbol{\theta}}_T = \boldsymbol{b}.
	$
  If $ \boldsymbol{Y} $ has an inverse, we obtain the best estimator
	\[ 
	\hat{\boldsymbol{\theta}}_T  = \boldsymbol{Y}^{-1}\boldsymbol{b}.
	\]

  It is crucial to recognize that the notion of ``best'' is inherently tied to the norm $\norm \cdot_D$ as defined initially. Proposition \ref{prop:quadraticform} establish that $\norm \cdot_D$ is indeed a norm. Furthermore, Theorem \ref{thm:maintheorem} establish that $\mathbb{E}[\gamma_T(f)]$ can be expressed as $\norm{ f - s}_D^2 - \norm s_D^2$. Within this framework, $\gamma_T(f)$ serves as an empirical representation of $\norm{ f - s}_D^2 - \norm s_D^2$, aligning with the conventional approach in Least-Squares Contrasts. Consequently, the estimator $\hat{\boldsymbol{\theta}}_T$ is optimized to minimize $\gamma_T(f)$ under the norm  $\norm \cdot_D$, thereby qualifying as the ``best'' estimator.

Building upon the properties of the contrast function, we now establish the key large-sample properties of the resulting least-squares estimator, $\hat{\boldsymbol{\theta}}_T$. These properties, namely consistency and asymptotic normality, are fundamental for statistical inference and validate the simulation results presented in Sections 3 and 4.

\begin{theorem}[Consistency of the LSE]
  \label{thm:consistency}
The least-squares estimator $\hat{\boldsymbol{\theta}}_T$ is consistent for the true parameter vector $s$. That is, under suitable regularity conditions, as the sample size $T \to \infty$
\[
\hat{\boldsymbol{\theta}}_T \xrightarrow{p} s
\]
\end{theorem}

\begin{proof}
  We provide the proof in \ref{sec:appendix_proof_consistency}. 
\end{proof}

Consistency ensures that with a sufficiently large amount of data, our estimator will be arbitrarily close to the true parameter values, providing a fundamental justification for the estimation method. 

\begin{theorem}[Approximate Normality of the LSE for Large $T$]
\label{thm:asymptotic_normality}
For a sufficiently large sample size $T$, under suitable regularity conditions, the distribution of the least-squares estimator $\hat{\boldsymbol{\theta}}_T$ can be approximated by a normal distribution:
\[
\sqrt{T}(\hat{\boldsymbol{\theta}}_T - s) \overset{\cdot}{\sim} N(\mathbf{0}, \boldsymbol{\Sigma}_T)
\]
where $\boldsymbol{\Sigma}_T = \mathbf{J}_T^{-1} \mathbf{K}_T \mathbf{J}_T^{-1}$ is the $T \times T$ sandwich covariance matrix, and the matrices $\mathbf{K}_T$ and $\mathbf{J}_T$ are positive definite.
\end{theorem}

\begin{proof}
  We provide the proof in \ref{sec:appendix_proof_asymptotic_normality}.
\end{proof}

The approximate covariance matrix has the celebrated ``sandwich'' form, $\boldsymbol{\Sigma}_T = \mathbf{J}_T^{-1} \mathbf{K}_T \mathbf{J}_T^{-1}$. Intuitively, for a large sample size $T$, the two outer matrices of the sandwich, $\mathbf{J}_T = 2\cdot\mathbb{E}[\mathbf{Y}]$, represent an approximation of the objective function's curvature, related to the expected Hessian. The inner matrix, or ``meat'' of the sandwich, $\mathbf{K}_T = \text{Var}(\sqrt{T}\nabla\gamma_{T}(s))$, represents the variance of the scaled score vector (the gradient evaluated at the true parameter) and captures the randomness from the process innovations for that given sample size $T$. This theoretical result provides the foundation for constructing confidence intervals and hypothesis tests.

We will further substantiate the practical efficacy of this estimation technique through numerical experiments.

  \section{Numerical Experiments: Consistency of the estimator}\label{sec:numericalResults}

  In this section, we illustrate the performance of the proposed least-squares 
estimator for the cumulative INAR($\infty$) [c-INAR($\infty$)] model via 
numerical experiments. Our simulation procedure is structured in two main parts. 
First, we describe the process of generating a single realization of the 
model, which is detailed in Algorithm~\ref{alg:inar_simulation}. 
Building upon this, we then present our comprehensive Monte Carlo framework in 
Algorithm~\ref{alg:monte_carlo}, which explains how multiple independent 
realizations are used to evaluate the statistical properties of the least-squares 
estimator (LSE). Finally, we present and discuss the results of these experiments, 
focusing on the estimator's accuracy.
  
  \subsection{Simulation of a Single Realization}
  We begin by simulating one path of length $T$ from a c-INAR($T-1$) process. Let $\nu>0$ be the immigration rate, and let $\alpha(\cdot)$ be the offspring function such that
  \[
  X_n \mid \mathcal{F}_{n-1} \;\sim\; \mathrm{Poisson}\!\Bigl(\,\nu + \sum_{k=1}^{n-1}\alpha(n-k)\,X_k\Bigr).
  \]
  Algorithm~\ref{alg:inar_simulation} outlines this procedure in detail.
  
  \begin{algorithm}[h]
  \caption{Simulating a single c-INAR($\infty$) realization}
  \label{alg:inar_simulation}
  \begin{algorithmic}[1]
      \Require Sample size $T$; true immigration rate $\nu$; offspring function $\alpha(\cdot)$.
      \State Initialize an array $X$ of length $T$ to store the realization.
      \State Draw $X_1 \sim \mathrm{Poisson}(\nu)$.

      \For{$n = 2 \to T$}
          \State Compute $\lambda_n \gets \nu + \sum_{k=1}^{n-1} \alpha(n-k) X_k$. 
          \State Sample $X_n \sim \mathrm{Poisson}(\lambda_n)$.
      \EndFor
      \Ensure The sequence $(X_n)_{1\le n\le T}$.
  \end{algorithmic}
  \end{algorithm}
  
  \subsection{Multiple Replications and Least-Squares Estimation}
  To evaluate the performance of the proposed CLS estimator in a finite-sample context, we conduct a comprehensive Monte Carlo simulation. The core idea is to generate a large number of independent sample paths from the process with known parameters. For each individual path, we compute a corresponding least-squares estimate. By analyzing the statistical properties of this collection of estimates, such as their mean and mean squared error, we can assess the estimator's accuracy and bias. This procedure allows us to verify whether the estimator behaves as predicted by the large-sample theory, even for a finite sample size $T$. The entire process for obtaining and evaluating the LSE is summarized in Algorithm \ref{alg:monte_carlo}.
  
  \begin{algorithm}[htbp]
\caption{Monte Carlo Simulation for Evaluating the LSE Performance}
\label{alg:monte_carlo}
\begin{algorithmic}[1]
\State \textbf{Require:} Number of replications $N_{\text{experiments}}$; sample size $T$; true parameters $s = (\nu, \alpha(\cdot))$.

\Statex \textit{// Part A: Generating a Collection of Estimates}
\State Initialize an empty list to store the results: $\texttt{estimator\_list} \leftarrow []$.
\For{$i = 1$ \textbf{to} $N_{\text{experiments}}$}
    \State Generate a single, independent sample path $X^{(i)} = (X_1^{(i)}, \dots, X_T^{(i)})$ ~\ref{alg:inar_simulation}.
    \State Construct the matrix $\mathbf{Y}^{(i)}$ and vector $\mathbf{b}^{(i)}$ based solely on the path $X^{(i)}$.
    \State Solve the linear system $\mathbf{Y}^{(i)} \hat{\theta}_T^{(i)} = \mathbf{b}^{(i)}$ to obtain the estimate $\hat{\theta}_T^{(i)}$.
    \State Append the resulting estimator $\hat{\theta}_T^{(i)}$ to $\texttt{estimator\_list}$.
\EndFor

\Statex \textit{// Part B: Analyzing the Estimator's Properties}
\State Let $\left(\hat{\theta}_T^{(i)}\right)_{i=1}^{N_{\text{experiments}}}$ be the collection of estimates in $\texttt{estimator\_list}$.
\State Compute the mean of the estimators to assess bias: $\bar{\theta}_T \leftarrow \frac{1}{N_{\text{experiments}}} \sum_{i=1}^{N_{\text{experiments}}} \hat{\theta}_T^{(i)}$.
\State Compute the Mean Squared Error (MSE) against the true parameter $s$: 
\Statex \qquad $\text{MSE} \leftarrow \frac{1}{N_{\text{experiments}}} \sum_{i=1}^{N_{\text{experiments}}} \|\hat{\theta}_T^{(i)} - s\|^2$.
\State Analyze the empirical distribution of the components of $\hat{\theta}_T^{(i)}$ (for histograms, Q-Q plots, and normality tests as in Section 4).

\Statex
\State \textbf{Output:} Statistical properties of the LSE (e.g., mean, MSE, empirical distribution).
\end{algorithmic}
\end{algorithm}
  
\subsection{Numerical Experiments}

To empirically validate the theoretical properties of our proposed Conditional Least-Squares (CLS) estimator, we conduct a series of Monte Carlo simulations. The experiments are designed to investigate two key aspects: (1) the estimator's performance and convergence as the sample size $T$ increases, and (2) its robustness under different parameter settings. All simulations follow the framework described in Algorithm~\ref{alg:monte_carlo} with $N_{\text{experiments}} = 1000$ replications for each setting. The number of estimated autoregressive parameters is fixed at $p=10$.

\subsubsection{Performance under Theoretical Assumptions}

We first analyze a scenario that fully satisfies the theoretical conditions of our framework, particularly $\|\alpha\|_1 < 1$ and $\|\alpha\|_2^2 < 1/2$. The true parameters are set to:
\begin{itemize}
    \item \textbf{Case 1}: $\nu=100$ and $\alpha_n=(1/4)^n$ for $n \geq 1$.
\end{itemize}

We evaluate the estimator's performance across three sample sizes: $T=200$, $T=500$, and $T=1000$. Table~\ref{tab:case1_results} summarizes the key performance metrics, averaged over all replications.

\begin{table}[htbp]
  \centering
  \caption{Estimator Performance for Case 1 ($\alpha_n = (1/4)^n$) across Different Sample Sizes ($T$)}
  \label{tab:case1_results}
  \resizebox{\textwidth}{!}{%
  \begin{tabular}{l c c c}
    \toprule
    \textbf{Metric / Parameter} & \textbf{T = 200} & \textbf{T = 500} & \textbf{T = 1000} \\
    \midrule
    \multicolumn{4}{l}{\textit{Mean of Parameter Estimates (Bias Assessment)}} \\
    \quad Mean($\hat{\nu}$) & 100.58 & 100.47 & 100.26 \\
    \quad Mean($\hat{\alpha}_1$) & 0.2486 & 0.2472 & 0.2489 \\
    \quad Mean($\hat{\alpha}_2$) & 0.0562 & 0.0600 & 0.0601 \\
    \midrule
    \multicolumn{4}{l}{\textit{Error Metrics (Variance and Accuracy Assessment)}} \\
    \quad Mean Squared Error (MSE) & 52.81 & 39.94 & 29.94 \\
    \quad Relative $\ell^2$-Error ($\boldsymbol{\theta}$) & 0.576\% & 0.466\% & 0.263\% \\
    \quad Relative $\ell^2$-Error ($\boldsymbol{\alpha}$) & 3.320\% & 1.790\% & 1.459\% \\
    \bottomrule
  \end{tabular}}
  \par\medskip
  \small \textit{Note: The true values are $\nu=100, \alpha_1=0.25, \alpha_2=0.0625$. The relative $\ell^2$-error for $\boldsymbol{\theta}$ is calculated as $\|\text{mean}(\hat{\boldsymbol{\theta}}) - \boldsymbol{\theta}_{\text{true}}\|_2 / \|\boldsymbol{\theta}_{\text{true}}\|_2$.}
\end{table}

The results in Table~\ref{tab:case1_results} provide strong empirical support for our theory. The mean of the parameter estimates remains very close to the true values across all sample sizes, suggesting that the CLS estimator is approximately unbiased. More importantly, we observe a clear trend of decreasing error metrics as $T$ grows. The Mean Squared Error (MSE), which captures both bias and variance, systematically declines from 52.81 at $T=200$ to 29.94 at $T=1000$. The relative $\ell^2$-errors for both the full parameter vector $\boldsymbol{\theta}$ and its autoregressive part $\boldsymbol{\alpha}$ show a consistent reduction. This empirically validates the consistency of the estimator, as established in Theorem 2.

\subsubsection{Robustness to Assumption Violations}
Next, we assess the estimator's performance when the theoretical condition $\|\alpha\|_2^2 < 1/2$ for our variance bounds is not met. We consider a case with a more concentrated autoregressive effect:
\begin{itemize}
    \item \textbf{Case 2}: $\nu=100$, $\alpha_1=0.8$, and $\alpha_n=0$ for $n \geq 2$. Here, $\|\alpha\|_2^2 = 0.64 > 0.5$.
\end{itemize}

The simulation results for Case 2, also across $T=200, 500, 1000$, are presented in Table~\ref{tab:case2_results}.

\begin{table}[htbp]
  \centering
  \caption{Estimator Performance for Case 2 ($\alpha_1=0.8$) across Different Sample Sizes ($T$)}
  \label{tab:case2_results}
  \resizebox{\textwidth}{!}{%
  \begin{tabular}{l c c c}
    \toprule
    \textbf{Metric / Parameter} & \textbf{T = 200} & \textbf{T = 500} & \textbf{T = 1000} \\
    \midrule
    \multicolumn{4}{l}{\textit{Mean of Parameter Estimates (Bias Assessment)}} \\
    \quad Mean($\hat{\nu}$) & 101.49 & 101.03 & 100.83 \\
    \quad Mean($\hat{\alpha}_1$) & 0.7967 & 0.7971 & 0.7990 \\
    \quad Mean($\hat{\alpha}_2$) & -0.0059 & -0.0017 & -0.0021 \\
    \midrule
    \multicolumn{4}{l}{\textit{Error Metrics (Variance and Accuracy Assessment)}} \\
    \quad Mean Squared Error (MSE) & 86.39 & 65.48 & 50.11 \\
    \quad Relative $\ell^2$-Error ($\boldsymbol{\theta}$) & 1.486\% & 1.031\% & 0.832\% \\
    \quad Relative $\ell^2$-Error ($\boldsymbol{\alpha}$) & 1.291\% & 0.789\% & 0.674\% \\
    \bottomrule
  \end{tabular}}
  \par\medskip
  \small \textit{Note: The true values are $\nu=100, \alpha_1=0.8, \alpha_2=0.0$. Negative estimates for $\alpha_n$ were capped at 0 before calculating error metrics, as per the discussion in the text.}
\end{table}

Several interesting observations emerge from Table~\ref{tab:case2_results}. First, the estimator demonstrates remarkable robustness. Even though the technical assumption is violated, the mean of the dominant parameter estimate, $\hat{\alpha}_1$, remains exceptionally close to its true value of 0.8. The estimates for other $\alpha_n$ coefficients are correctly centered around zero.

Second, the convergence property is preserved. The MSE and relative errors again decrease monotonically as the sample size $T$ increases, reinforcing the estimator's consistency. Interestingly, the relative error for the $\boldsymbol{\alpha}$ vector is lower in this case compared to Case 1. This is because the signal is concentrated in a single, large coefficient, making it easier to distinguish from noise compared to the distributed, decaying coefficients in Case 1.

\begin{remark}
These experiments collectively demonstrate that the CLS estimator is a practical and reliable tool for INAR($\infty$) processes. The results not only confirm its consistency by showing clear convergence as sample size increases under different settings but also highlight its robustness to violations of certain technical assumptions. The analysis underscores the finite-sample trade-off between bias and variance, providing valuable insights for practical applications.
\end{remark}

  \section{Numerical Experiment: Asymptotic Normality in Finite Samples}
\label{sec:asymptotic_normality}

Beyond convergence, a crucial property for statistical inference is the distribution of the estimator. In the traditional sense, \emph{asymptotic normality} refers to an estimator's distribution converging to a normal distribution as the sample size approaches infinity. To bridge the gap between theory and practice for our c-INAR($\infty$) model—where the estimator's dimension can grow with the sample size—we numerically investigate the finite-sample distribution of the leading components of $\hat{\boldsymbol{\theta}}_T$.

The goal is to assess how well the empirical distributions conform to their theoretical normal counterparts at practical sample sizes, specifically $T=200$ and $T=500$. The methodology follows the Monte Carlo framework (Algorithm~\ref{alg:monte_carlo}) for \textbf{Case 1} ($\nu=100, \alpha_n=(1/4)^n$), as this scenario satisfies our theoretical assumptions. For each sample size, we generate 1000 independent estimates ($\hat{\boldsymbol{\theta}}_T^{(i)}$) and analyze their distributions using visual tools (histograms, Q-Q plots) and formal statistical tests (Jarque-Bera, Shapiro-Wilk).

\subsection{Illustrative Figures and Observations}

The results of our normality investigation are presented in Figure~\ref{fig:combined_figures} and Table~\ref{tab:test_results}. A visual inspection of the figures immediately reveals the impact of the sample size: the distributions at $T=500$ appear more symmetric and concentrated than at $T=200$. The formal test results quantify these visual observations.

\begin{figure}[h!]
  \centering
  \begin{subfigure}[b]{0.48\textwidth}
    \centering
    \includegraphics[width=\textwidth]{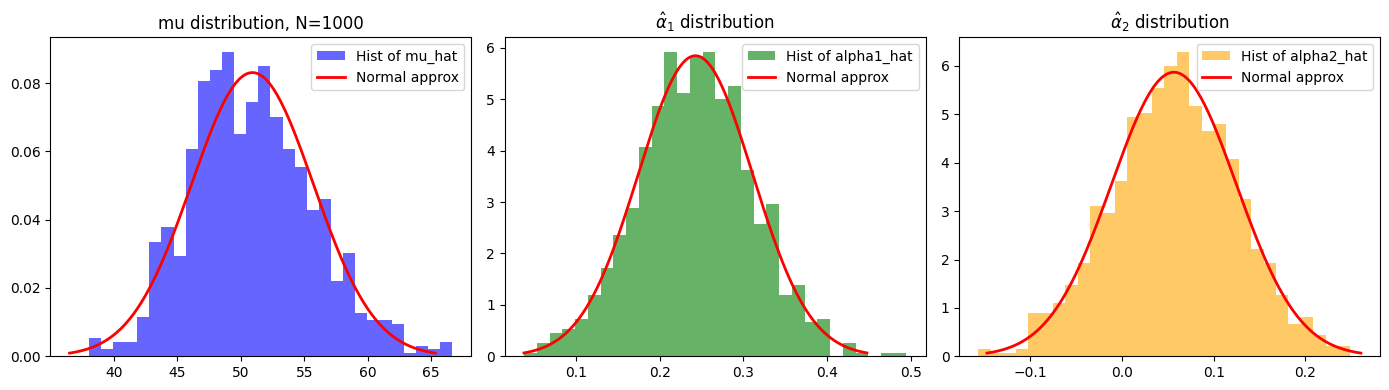}
    \caption{Histograms of LSE components at $T=200$.}
    \label{fig:hist_t200}
  \end{subfigure}
  \hfill
  \begin{subfigure}[b]{0.48\textwidth}
    \centering
    \includegraphics[width=\textwidth]{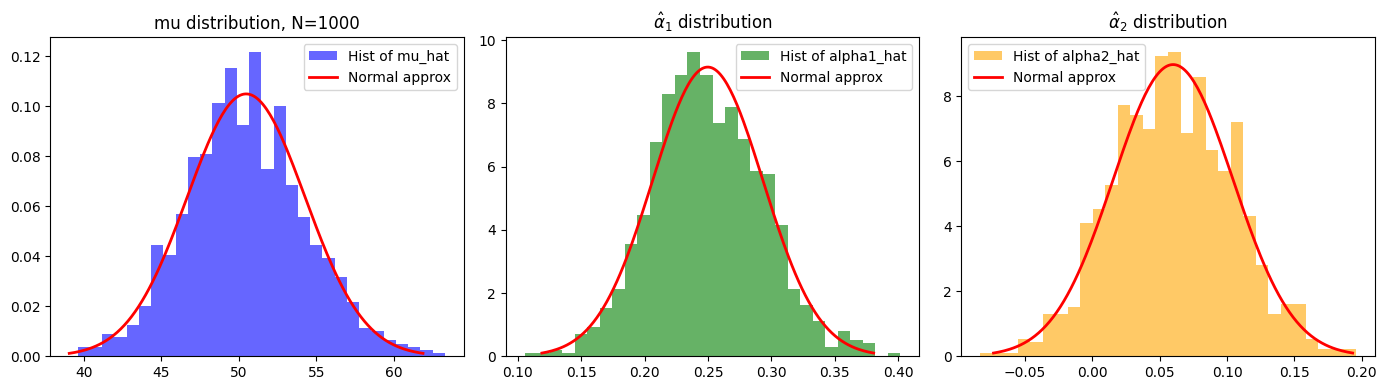}
    \caption{Histograms of LSE components at $T=500$.}
    \label{fig:hist_t500}
  \end{subfigure}
  
  \vspace{0.5cm}
  
  \begin{subfigure}[b]{0.48\textwidth}
    \centering
    \includegraphics[width=\textwidth]{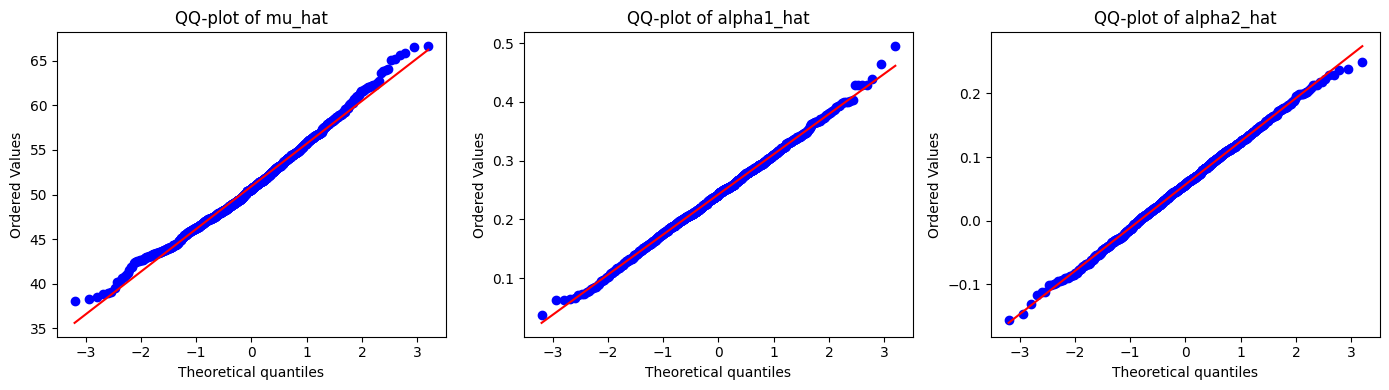}
    \caption{Q-Q plots for LSE components at $T=200$.}
    \label{fig:qq_t200}
  \end{subfigure}
  \hfill
  \begin{subfigure}[b]{0.48\textwidth}
    \centering
    \includegraphics[width=\textwidth]{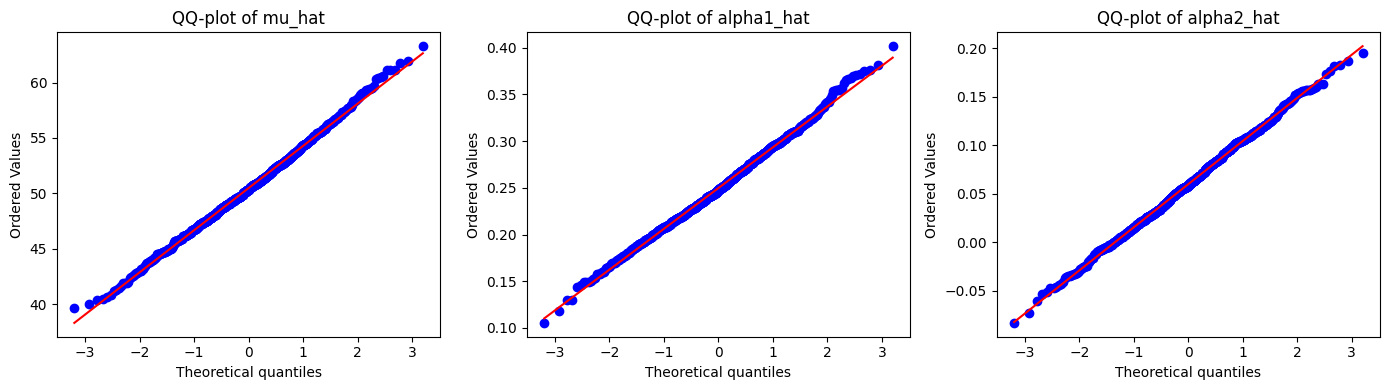}
    \caption{Q-Q plots for LSE components at $T=500$.}
    \label{fig:qq_t500}
  \end{subfigure}
  
  \caption{Histograms and Q-Q plots of the least-squares estimator (LSE) components for sample sizes $T=200$ and $T=500$. The plots show distributions for $\hat{\nu}$ (blue), $\hat{\alpha}_1$ (green), and $\hat{\alpha}_2$ (orange).}
  \label{fig:combined_figures}
\end{figure}

\begin{table}[h!]
  \centering
  \caption{Jarque-Bera and Shapiro-Wilk Test p-values for LSE Components at $T=200$ and $T=500$}
  \label{tab:test_results}
  \begin{tabular}{l|cc|cc}
    \toprule
    \multirow{2}{*}{\textbf{Parameter}} & \multicolumn{2}{c|}{\textbf{$T=200$}} & \multicolumn{2}{c}{\textbf{$T=500$}} \\
    \cmidrule(lr){2-3} \cmidrule(lr){4-5}
     & \textbf{JB p-value} & \textbf{SW p-value} & \textbf{JB p-value} & \textbf{SW p-value} \\
    \midrule
    $\hat{\nu}$         & 0.0002 & 0.1106 & 0.1174 & 0.4589 \\
    $\hat{\alpha}_1$   & 0.7372 & 0.6985 & 0.0872 & 0.2456 \\
    $\hat{\alpha}_2$   & 0.3988 & 0.6192 & 0.5363 & 0.4089 \\
    \bottomrule
  \end{tabular}
\end{table}

At the smaller sample size of $T=200$, the results are mixed. For the immigration rate estimator, $\hat{\nu}$, we observe a conflict between the tests: the Jarque-Bera test strongly rejects normality ($p=0.0002$), suggesting the presence of skewness or heavy tails, while the Shapiro-Wilk test does not ($p=0.1106$). This discrepancy indicates that the distribution has not fully converged. In contrast, the autoregressive estimators $\hat{\alpha}_1$ and $\hat{\alpha}_2$ appear reasonably normal even at this sample size, with high p-values from both tests.

The situation improves markedly at $T=500$. All estimators for $\hat{\nu}$, $\hat{\alpha}_1$, and $\hat{\alpha}_2$ now pass both normality tests comfortably, with all p-values well exceeding the 0.05 significance level. This, combined with the more symmetric histograms and better-aligned Q-Q plots in Figure~\ref{fig:combined_figures}, provides strong empirical evidence that the sampling distributions are indeed converging towards normality as the sample size increases.

\begin{remark}
The numerical analysis confirms the asymptotic normality property of the CLS estimator in a practical, finite-sample context. While smaller sample sizes like $T=200$ may exhibit some deviation from normality (particularly for the intercept term), the approximation becomes robust as the sample size grows. This finding complements our earlier results on consistency and solidifies the foundation for using this estimator for statistical inference, such as constructing confidence intervals and conducting hypothesis tests, in sufficiently large datasets.
\end{remark}

\section{Concluding Remarks}
\label{sec:conclusion}

In this paper, we have developed a comprehensive framework for the estimation and inference of the cumulative INAR($\infty$) process, a vital model for count time series that is equivalent to discrete-time Hawkes processes. Our primary contribution is the introduction and rigorous analysis of a computationally efficient conditional least-squares (CLS) estimator, particularly within a high-dimensional setting where the number of parameters is allowed to grow with the sample size.

Our theoretical investigation established the fundamental properties of the CLS estimator, including its consistency and asymptotic normality. A key theoretical result is the derivation of the sandwich-form covariance matrix for the estimator, which correctly accounts for the underlying conditional Poisson structure of the process and is crucial for accurate statistical inference.

These theoretical findings were substantiated through extensive Monte Carlo simulations. The numerical experiments provided strong empirical evidence for our theory, demonstrating two key results:
\begin{enumerate}
    \item \textbf{Consistency in Practice:} The estimator's accuracy and stability, as measured by Mean Squared Error and relative errors, systematically improve as the sample size increases from $T=200$ to $T=1000$. This holds true both when the model's theoretical assumptions are met and when they are violated, highlighting the estimator's robustness.
    \item \textbf{Convergence to Normality:} The finite-sample distribution of the estimator progressively converges to a normal distribution with larger sample sizes, confirming the practical applicability of our asymptotic normality results for constructing confidence intervals and performing hypothesis tests.
\end{enumerate}

While our proposed CLS estimator proves to be effective and theoretically sound, our work also illuminates potential avenues for future research. The observed variance of the estimator in smaller samples suggests that incorporating regularization techniques, such as Ridge or LASSO penalties, could enhance its finite-sample stability. Furthermore, extending this accessible least-squares framework to multivariate INAR($\infty$) or other complex point process models remains a promising direction.

In summary, this work provides a complete and practical toolkit for the statistical analysis of c-INAR($\infty$) processes. It not only offers a solid theoretical foundation but also delivers clear empirical validation, paving the way for its reliable application in fields such as finance, epidemiology, and social sciences where self-exciting count data are prevalent.

\section{Contribution}
    \textbf{Ying-Li Wang:} Conceptualization, Methodology, Software, Formal analysis, Investigation, Validation, Visualization, Writing – original draft.
    
    \textbf{Xiao-Hong Duan:} Data curation, Software, Validation.
    
    \textbf{Ping He:} Supervision, Project administration.
  
\appendix
\section{Proofs}
\label{sec:appendix_proofs}
\subsection{Proof of Lemma \ref{discreterenewalequation}}
\label{sec:appendix_proof_discrete_renewal_equation}
  Consider the generating functions (or \( z \)-transforms) of the sequences involved:
  \[
    \mathcal{G}(x)(z) = \sum_{n=1}^\infty x_n z^n, \quad \mathcal{G}(y)(z) = \sum_{n=1}^\infty y_n z^n, \quad \mathcal{G}(\eta)(z) = \sum_{n=1}^\infty \eta_n z^n.
  \]
  Taking the \( z \)-transform on both sides of equation \eqref{seriesconvolution}, and using the convolution property of \( z \)-transforms \citet{oppenheim1999discrete}, we obtain
  \[
    \mathcal{G}(x)(z) = \mathcal{G}(y)(z) + \mathcal{G}(\eta)(z) \cdot \mathcal{G}(x)(z).
  \]
  Solving for \( \mathcal{G}(x)(z) \), we get
  \[
    \mathcal{G}(x)(z) \left(1 - \mathcal{G}(\eta)(z)\right) = \mathcal{G}(y)(z),
  \]
  which leads to
  \[
    \mathcal{G}(x)(z) = \frac{\mathcal{G}(y)(z)}{1 - \mathcal{G}(\eta)(z)}.
  \]
  This step utilizes the property that if \( \left| \mathcal{G}(\eta)(z) \right| < 1 \), the above equation holds \citet{edition2002probability}.

  To find \( x_n \), we perform the inverse \( z \)-transform on both sides. The expression
  \[
    \frac{1}{1 - \mathcal{G}(\eta)(z)}
  \]
  corresponds to the generating function of the convolution inverse sequence \( (A_n)_{n\ge1} \). Therefore, by the convolution theorem \citet{oppenheim1999discrete}, we obtain
  \[
    x_n = y_n + \sum_{i=1}^{n-1} A_i y_{n-i}.
  \]
  This concludes the proof.

\subsection{Proof of Lemma \ref{boundlemma}}
\label{sec:appendix_proof_boundlemma}
  First, by the Cauchy-Schwarz inequality,
  \[
    \left( \sum_{k=1}^{n-1}\beta_k^{\frac12}\beta_k^{\frac12}X_{n-k} \right)^2
    \le \left(\sum_{k=1}^{n-1}\beta_k\right)\left( \sum_{k=1}^{n-1}\beta_kX_{n-k}^2 \right)
    =\sum_{k=1}^{n-1}\beta_k\sum_{\tau=1}^{n-1}\beta_\tau X_{n-\tau}^2,
  \]
  taking the expectation of both sides yields
  \begin{align*}
    \mathbb E\left[ \left( \sum_{k=1}^{n-1}\beta_kX_{n-k} \right)^2 \right]
    \le &\mathbb E \left[ \left( \sum_{k=1}^{n-1}\beta_k\sum_{\tau=1}^{n-1}\beta_\tau X_{n-\tau}^2 \right) \right]\\
    =&\sum_{k=1}^{n-1}\beta_k\sum_{\tau=1}^{n-1}\beta_{n-\tau}\mathbb E[X_\tau^2]\\
    \le&  \frac{2\nu^2(1-\norm \alpha_1)+\nu}{(1-2\norm\alpha_2^2)(1-\norm\alpha_1)}\left( \sum_{k=1}^{n-1} \beta_k \right)^2.
  \end{align*}

\subsection{Proof of Proposition \ref{prop:quadraticform}}
\label{sec:appendix_proof_quadraticform}
  Assume $f=(\mu,\beta)$, we will compute $\norm f_D^2$,
  \begin{equation}\label{f_D2}
    \begin{aligned}
    \norm f_D^2
    = &\mathbb{E}[D_T^2(f)] =\frac{1}{T}\sum_{n=1}^{T}\mathbb{E}\left[\Phi^2_f(n)\right] \\
    = & \frac{1}{T}\sum_{n=1}^{T}\mathbb{E}\left[ \left( \mu + \sum_{k=1}^{n-1}\beta_kX_{n-k} \right) ^2 \right]  \\
    = & \frac{1}{T}\sum_{n=1}^{T}\mathbb{E}\left[\mu ^2 + 2\mu \sum_{k=1}^{n-1}\beta_kX_{n-k} + \left( \sum_{k=1}^{n-1}\beta_kX_{n-k} \right) ^2  \right].
    \end{aligned}
    \end{equation}
  It is easy to verify $\forall f=(\mu,\beta), g=(\lambda,\xi)\in \mathfrak{l}^2$,
  \[
    \frac12(\norm {f+g}_D^2-\norm {f}_D^2-\norm g_D^2)=\frac1T\mathbb E\left[ \sum_{n=1}^T\Phi_f(n)\Phi_g(n) \right],
  \]
  and $\norm f_D^2=0$ if and only if $f=0$. Next, let's prove $\norm \cdot_D$ is equivalent to $\norm \cdot_2$, i.e. \eqref{equivalentnorm}. For the lower bound, we rewrite \eqref{f_D2}, the RHS equals
  \begin{equation}\label{expression1}
    \frac1T\sum_{n=1}^T\left( \mu+\mathbb E\left[ \sum_{k=1}^{n-1}\beta_kX_{n-k} \right]\right)^2+\text{Var}\left[ \sum_{k=1}^{n-1}\beta_kX_{n-k} \right].
  \end{equation}
  For the first part, note that $\mathbb E[X_n]\ge\nu$, for $\theta\in(0,1)$,
  \begin{equation*}
     \begin{aligned} 
    &\frac1T\sum_{n=1}^T\left( \mu+\mathbb E\left[ \sum_{k=1}^{n-1}\beta_kX_{n-k} \right]\right)^2\\
    \ge &\frac1T\sum_{n=1}^T\left( \mu+ \nu\sum_{k=1}^{n-1}\beta_k \right)^2\\
    \ge &\frac1T\sum_{n=1}^T\left((1-\theta)\mu^2+(1-\frac1\theta)\nu^2\left( \sum_{k=1}^{n-1}\beta_k \right)^2\right)\\
    \ge &(1-\theta)\mu^2+\frac1T(1-\frac1\theta)\nu^2\sum_{n=1}^T(n-1)\sum_{k=1}^{n-1}\beta_k^2,
  \end{aligned}
\end{equation*}
  where the second inequality is obviously established since $\mu,\nu,\beta_k\ge0$.

  For the second part, consider first a continuous-time Hawkes process $(\tilde N_t)_{t\ge0}$ with exciting function \eqref{generalizedexpressionofdiscrete}. From \citet{bremaud2001hawkes}, for any $\phi\in L^1\cap L^2$,
  \begin{equation}\label{bartlettformula}
    \text{Var}\left[ \int_\mathbb R\phi(u)d\tilde{N}_u \right] 
    = \int_\mathbb R|\hat\phi(\omega)|^2f_{\tilde{N}}(\omega)d\omega
  \end{equation}
  where $\hat \phi$ is the Fourier transform of $\phi$, $\hat \phi(\omega) = \int_\mathbb R e^{\text{i}\omega t}\phi(t)dt$, $f_{\tilde N}$ is the Bartlett spectrum density of continuous-time Hawkes process $\tilde N$. Since the Fourier transform of $h$ is 
  \[
    \hat h(\omega) 
    = \sum_{k=1}^\infty \alpha_k\int_\mathbb R e^{\text{i}\omega t}\delta_{\{t=k\}}dt
    = \sum_{k=1}^\infty \alpha_ke^{\text{i}\omega k},
  \]
  \begin{align*}
    f_{\tilde N}(\omega)
    =\frac{\nu}{2\pi (1-\norm \alpha_1)|1-\hat h(\omega)|^2}=\frac{\nu}{2\pi (1-\norm \alpha_1)|1-\sum_{k=1}^\infty \alpha_ke^{\text{i}\omega k}|^2}.
  \end{align*}
  Given $n\in\mathbb N$, let 
  \[
    \phi(t)=\phi_n(t):=\beta_{n-\lfloor t \rfloor-1}1_{\{0<t<n\}}=\beta_{\lfloor n-t \rfloor}1_{\{t<n\}}=g(n-t)1_{\{t<n\}},
  \]
  set $\beta_0=0$ for convenience, since $g$ has a positive support, $\hat\phi(\omega)=e^{\text{i}\omega t}\hat g(-\omega)$. Hence, 
  \[
    \text{Var}\left[ \int_\mathbb R\phi(u)d\tilde{N}_u \right] = \int_\mathbb R|\hat g(-\omega)|^2f_{\tilde N}(\omega)d\omega.
  \]
   Since $f_{\tilde N}(\omega)\ge\frac{\nu}{2\pi(1-\norm \alpha_1)(1+\norm \alpha_1)^2}$, and due to the Plancherel's identity, i.e. 
   \[
     \int_\mathbb R|\hat g(-\omega)|^2d\omega=2\pi \sum_{k=1}^{n-1} \beta_k^2,
  \]
  we obtain 
  \[
    \text{Var}\left[ \int_\mathbb R\phi(u)d\tilde{N}_u \right]\ge \frac{\nu}{(1-\norm\alpha_1)(1+\norm \alpha_1)^2}\sum_{k=1}^{n-1} \beta_k^2.
  \]
  Hence, set $c=\frac{\nu}{2\pi(1-\norm\alpha_1)(1+\norm\alpha_1)^2}$, 
  \begin{align*}
    \text{Var}\left[ \sum_{u=1}^{n-1}\beta_{n-u}X_u \right]
    =&\text{Var}\left[ \int_\mathbb R\beta_{n-\lfloor u \rfloor-1}1_{\{u<n\}}d\tilde N_u \right]\\
    =&\text{Var}\left[ \int_\mathbb R\phi(u)d\tilde{N}_u \right]\ge 2\pi c\sum_{k=1}^{n-1} \beta_k^2.
  \end{align*}
  Combine them together,
  \begin{equation*}
      \begin{aligned}
    \norm f_D^2
    \ge& (1-\theta)\mu^2+(1-\frac1\theta)\nu^2\frac1T\sum_{n=1}^T\left((n-1)\sum_{k=1}^{n-1}\beta_k^2+2\pi c\sum_{k=1}^{n-1}\beta_k^2\right)\\
    \ge& (1-\theta)\mu^2+\left[ (1-\frac1\theta)\nu^2\frac{T-1}{2}+\frac{2\pi c}{T} \right]\sum_{k=1}^{T-1}\beta_k^2.
  \end{aligned}
\end{equation*}
  Choose $\theta$ satisfying $(1-\frac1\theta)\nu^2\frac{T-1}{2}+\frac{2\pi c}{T}=\frac{\pi c}{T}$, i.e. 
  \[
    \theta = \frac{\nu T(T-1)(1+\norm \alpha_1)^2}{1+\nu T(T-1)(1+\norm \alpha_1)^2},
  \] then 
  \[\|f\|_D^2 \geq    \dfrac{1}{1+\nu T(T-1)(1+\| \alpha \|_1)^2}\mu^2 + 
     \dfrac{\nu  }{2 T (1-\| \alpha \|_1)(1+\| \alpha \|_1)^2}
     \sum_{k=1}^{T-1}\beta_k ^2.
  \]
  Finally we obtain 
  \[
    L^2 = \min \left\lbrace   \dfrac{1}{1+\nu T(T-1)(1+\| \alpha \|_1)^2},\dfrac{\nu  }{2 T (1-\| \alpha \|_1)(1+\| \alpha \|_1)^2}\right\rbrace.
  \]
  For the upper bound, from \eqref{expression1} we can see
  \begin{align*}
    \norm f_D^2 
    \le &\frac1T\sum_{n=1}^T\left\{ \left( \mu+\frac{\nu}{1-\norm \alpha_1}\sum_{k=1}^{n-1}\beta_k \right)^2+\mathbb E\left[ \left( \sum_{k=1}^{n-1}\beta_kX_{n-k} \right)^2 \right] \right\}.
  \end{align*}
  For the first term inside the curly braces on the RHS, it is bounded by the following
  \[
    \left( \mu+\frac{\nu}{1-\norm \alpha_1}\sum_{k=1}^{n-1}\beta_k \right)^2\le 2\mu^2+2\frac{\nu^2}{(1-\norm\alpha_1)^2}\left( \sum_{k=1}^{n-1}\beta_k \right)^2.
  \]
  By Lemma \ref{boundlemma}, 
  \[
    \mathbb{E}\left[ \left( \sum_{k=1}^{n-1} \beta_k X_{n-k} \right)^2 \right] \leq \frac{2\nu^2(1-\norm \alpha_1)+\nu}{(1-2\norm\alpha_2^2)(1-\norm\alpha_1)} \left( \sum_{k=1}^{n-1} \beta_k \right)^2.
  \]
  Hence,
 \begin{equation*}
    \begin{aligned}
    \norm f_D^2
    \le &2\mu^2+\left[ \frac{2\nu^2}{(1-\norm \alpha_1)^2}+\frac{2\nu^2(1-\norm\alpha_1)+\nu}{(1-2\norm \alpha_2^2)(1-\norm\alpha_1)} \right]\cdot \frac1T\sum_{n=1}^T\left( \sum_{k=1}^{n-1}\beta_k \right)^2\\
    \le &2\mu^2+\left[ \frac{2\nu^2}{(1-\norm \alpha_1)^2}+\frac{2\nu^2(1-\norm\alpha_1)+\nu}{(1-2\norm \alpha_2^2)(1-\norm\alpha_1)} \right]\frac1T\sum_{n=1}^T(n-1)\sum_{k=1}^{n-1}\beta_k^2\\
    \le &2\mu^2+\left[ \frac{2\nu^2}{(1-\norm \alpha_1)^2}+\frac{2\nu^2(1-\norm\alpha_1)+\nu}{(1-2\norm \alpha_2^2)(1-\norm\alpha_1)} \right]\left(\frac{T-1}{2}\right)\sum_{k=1}^{T-1}\beta_k^2.
    \end{aligned}
 \end{equation*}
  Finally we obtain,
  \[
    K^2 = \max\left\{ 2, \frac{T-1}{2}\left[ \frac{2\nu^2}{(1-\norm \alpha_1)^2}+\frac{2\nu^2(1-\norm\alpha_1)+\nu}{(1-2\norm \alpha_2^2)(1-\norm\alpha_1)} \right]\right\}.
  \]

\subsection{Proof of Theorem \ref{thm:maintheorem}}
\label{sec:appendix_proof_maintheorem}
  By the bilinear property of $D_T^2(f)$ and the Iterated expectation theorem, we obtain
  \begin{equation*}
    \begin{aligned}
    \mathbb{E}\left[\gamma_T(f)\right] = & \mathbb{E}\left[-\frac{2}{T}\sum_{n=1}^{T}\Phi_f(n)X_n\right] + \mathbb{E}\left[\frac{1}{T}\sum_{n=1}^{T}\Phi^2_f(n)\right] \\
    = & \mathbb{E}\left[-\frac{2}{T}\sum_{n=1}^{T}\Phi_f(n)\Phi_s(n)\right] + \mathbb{E}[D_T^2(f)] \\
    = & \mathbb{E}\left[-\frac{2}{T}\sum_{n=1}^{T}\Phi_f(n)\Phi_s(n)\right] + \|f\|_D^2\\
    = & \mathbb{E}\left[\frac{1}{T}\sum_{n=1}^{T}(\Phi_f(n)-\Phi_s(n))^2\right] - 
    \mathbb{E}\left[\frac{1}{T}\sum_{n=1}^{T}\Phi^2_s(n)\right]  \\
    = & \|f-s\|_D^2 - \|s\|_D^2.
    \end{aligned}
    \end{equation*}
  From Proposition \ref{prop:quadraticform}, $\norm \cdot_D$ is a norm. As a result, $\mathbb E[\gamma_T(f)]$ reaches its minimum when $f=s$.

\subsection{Proof of Theorem \ref{thm:consistency}}
\label{sec:appendix_proof_consistency}
The consistency of the least-squares estimator $\hat{\boldsymbol{\theta}}_T$ is established by verifying the conditions of the general theory for the consistency of extremum estimators (also known as M-estimators). We follow, for example, the framework laid out in \citet{newey1994large}. 

First, we state the necessary regularity conditions:
\begin{enumerate}
    \item \textbf{(A1)Identification:} The true parameter $s$ is the unique minimizer of the limiting objective function $Q(f) = \mathbb{E}[\gamma_T(f)]$ over the parameter space $\Theta$. This is satisfied by our Theorem 2.5.
    \item \textbf{(A2)Compactness:} We assume the parameter space $\Theta$ is a compact subset of $\mathfrak{l}^2$. 
    \item \textbf{(A3)Continuity:} The contrast function $\gamma_T(f)$ is a continuous function of $f \in \Theta$ for any given sample path. This is true by construction, as $\gamma_T(f)$ is a quadratic function of the parameters in $\theta$.
    \item \textbf{(A4)Uniform Convergence:} The sample contrast function $\gamma_T(f)$ converges uniformly in probability to its expectation $Q(f)$ over $\Theta$. That is:
    \[
      \sup_{f \in \Theta} | \gamma_T(f) - Q(f) | \xrightarrow{p} 0. 
    \]
    This condition is guaranteed by the Uniform Law of Large Numbers (ULLN), see e.g. \citet{peskir1994necessary}, which is the key assumption for proving consistency. Specifically, for stationary and ergodic sequences, as in our INAR($\infty$) model, \citet{peskir1994necessary} provide a set of necessary and sufficient conditions for the ULLN to hold. They show that properties such as ``eventual total boundedness in mean'' are equivalent to uniform convergence in probability, in mean, and almost surely. Our proof framework relies on these established theoretical results for stationary processes.
\end{enumerate}

\begin{proof}[Proof of Theorem \ref{thm:consistency}]
The proof proceeds by showing that the minimizer of $\gamma_T(f)$ must lie within an arbitrarily small neighborhood of $s$ as $T \to \infty$.

Let $N$ be an arbitrary open neighborhood of $s$ in $\Theta$. Let $N^c$ be the complement of $N$ in $\Theta$. Since $\Theta$ is compact and $N$ is open, $N^c$ is also compact.

From condition A1, we know that for any $f \in N^c$, $Q(f) > Q(s)$. Because $Q(f)$ is continuous and $N^c$ is compact, there exists a constant $\delta > 0$ such that $\inf_{f \in N^c} Q(f) \ge Q(s) + \delta$.

Now, consider the difference in the sample contrast function:
$$ \gamma_T(f) - \gamma_T(s) = (Q(f) - Q(s)) + (\gamma_T(f) - Q(f)) - (\gamma_T(s) - Q(s)). $$
Using the triangle inequality, we have:
$$ |(\gamma_T(f) - Q(f)) - (\gamma_T(s) - Q(s))| \le 2 \sup_{f \in \Theta} | \gamma_T(f) - Q(f) |. $$
From the ULLN (Condition A4), the right-hand side term converges to 0 in probability. This means that for a large enough $T$, the random term becomes negligible compared to the deterministic difference $Q(f) - Q(s)$.

Specifically, for any $f \in N^c$, we have $Q(f) - Q(s) \ge \delta$. With probability approaching 1, the random part will be smaller than $\delta/2$, which implies:
$$ \prob\left( \inf_{f \in N^c} \gamma_T(f) > \gamma_T(s) \right) \to 1 \quad \text{as } T \to \infty. $$
This statement means that, with probability approaching 1, the minimum value of the sample contrast function $\gamma_T$ over the set $N^c$ (everywhere outside the neighborhood of $s$) is strictly greater than its value at $s$.

Since the estimator $\hat{\boldsymbol{\theta}}_T$ is defined as the global minimizer of $\gamma_T(f)$ over $\Theta$, it must be that $\hat{\boldsymbol{\theta}}_T$ lies inside the neighborhood $N$ with probability approaching 1. As the neighborhood $N$ can be chosen to be arbitrarily small, this implies that $\hat{\boldsymbol{\theta}}_T$ converges in probability to $s$.

This formalizes the argument. The result is a direct application of, for example, Theorem 2.1 in \citet{newey1994large}.
\end{proof}

\subsection{Proof of Theorem \ref{thm:asymptotic_normality}}
\label{sec:appendix_proof_asymptotic_normality}
The proof relies on a first-order Taylor expansion of the estimator's First-Order Condition (FOC) and provides a heuristic justification for the normal approximation for large $T$.

\begin{enumerate}
    \item \textbf{First-Order Condition (FOC):} By definition, the LSE $\hat{\boldsymbol{\theta}}_T$ satisfies the FOC: 
    $$ \nabla \gamma_T(\hat{\boldsymbol{\theta}}_T) = \mathbf{0} $$

    \item \textbf{Taylor Expansion:} A mean-value expansion of the FOC around the true parameter $s$ gives:
    $$
    \mathbf{0} = \nabla \gamma_T(s) + \nabla^2 \gamma_T(\bar{\boldsymbol{\theta}}) (\hat{\boldsymbol{\theta}}_T - s)
    $$
    where $\bar{\boldsymbol{\theta}}$ is a point on the line segment between $\hat{\boldsymbol{\theta}}_T$ and $s$.

    \item \textbf{Rearrangement:} We can rearrange the expression to isolate the term of interest:
    $$
    \sqrt{T}(\hat{\boldsymbol{\theta}}_T - s) = - \left[ \nabla^2 \gamma_T(\bar{\boldsymbol{\theta}}) \right]^{-1} \left[ \sqrt{T} \nabla \gamma_T(s) \right]
    $$

    \item \textbf{Approximation of the Hessian (Definition of $\mathbf{J}_T$):}
    The Hessian matrix is $\nabla^2 \gamma_T(\boldsymbol{\theta}) = 2\mathbf{Y}$. By the Ergodic Theorem for stationary processes (see e.g. Theorem 24.1 in \citet{billingsley2012probability}), the sample matrix $\mathbf{Y}$ is a consistent estimator for its expectation, $\mathbb{E}[\mathbf{Y}]$. Since $\hat{\boldsymbol{\theta}}_T$ is consistent for $s$, $\bar{\boldsymbol{\theta}}$ is also consistent. We define the deterministic $T \times T$ matrix $\mathbf{J}_T$ as:
    $$
    \mathbf{J}_T := 2 \cdot \mathbb{E}[\mathbf{Y}]
    $$
    For large $T$, the random Hessian $\nabla^2 \gamma_T(\bar{\boldsymbol{\theta}})$ is thus close to $\mathbf{J}_T$ in probability.

    \item \textbf{Distribution of the Score (Definition of $\mathbf{K}_T$):}
    The scaled score vector, $\sqrt{T} \nabla \gamma_T(s)$, is the source of the randomness in the estimator. From the definition of $\gamma_T(f)$, its gradient with respect to the parameter vector $\boldsymbol{\theta}$ is $\nabla\gamma_T(\boldsymbol{\theta}) = 2(\mathbf{Y}\boldsymbol{\theta} - \mathbf{b})$. Evaluating this at the true parameter vector $\boldsymbol{\theta} = s$ gives $\nabla\gamma_T(s) = 2(\mathbf{Y}s - \mathbf{b})$. We can write the entire gradient vector compactly as:
    $$
    \nabla\gamma_T(s) = \frac{2}{T} \sum_{n=1}^T \mathbf{Z}_n \left( X_n - \Phi_s(n) \right)
    $$
    where $\mathbf{Z}_n = (1, X_{n-1}, X_{n-2}, \dots, X_1, 0, \dots)^\top$ is the vector of regressors available at time $n-1$.

    Let us define a vector sequence $\mathbf{d}_n = \mathbf{Z}_n (X_n - \Phi_s(n))$. This sequence forms a \textbf{martingale difference sequence (MDS)} with respect to the filtration $\mathcal{F}_{n-1}$, since:
    $$
    \mathbb{E}[\mathbf{d}_n | \mathcal{F}_{n-1}] = \mathbf{Z}_n \cdot \mathbb{E}[X_n - \Phi_s(n) | \mathcal{F}_{n-1}] = \mathbf{Z}_n \cdot 0 = \mathbf{0}.
    $$
    By applying a Martingale Central Limit Theorem (see e.g., \citet{hall2014martingale}) to the sum of this MDS, the scaled score vector is approximately normally distributed for large $T$:
    $$
    \sqrt{T} \nabla \gamma_T(s) \overset{\cdot}{\sim} N(\mathbf{0}, \mathbf{K}_T)
    $$
    where the matrix $\mathbf{K}_T$ is the $T \times T$ variance-covariance matrix of the scaled score vector, defined as:
    \[
    \mathbf{K}_T = \text{Var}(\sqrt{T}\nabla\gamma_{T}(s)).
    \]

    \item \textbf{Conclusion and Final Approximation:}
    We combine the results from the previous steps. By substituting the approximation for the Hessian and the approximate distribution for the score, we obtain the approximation for our estimator:
    $$
    \sqrt{T}(\hat{\boldsymbol{\theta}}_T - s) \approx - \mathbf{J}_T^{-1} \left[ \sqrt{T} \nabla \gamma_T(s) \right]
    $$
    Since the scaled score vector is approximately distributed as $N(\mathbf{0}, \mathbf{K}_T)$, its linear transformation by $-\mathbf{J}_T^{-1}$ is also approximately normal. The variance-covariance matrix of this resulting approximate distribution is:
    $$
    \text{Var}(-\mathbf{J}_T^{-1} \cdot N(\mathbf{0}, \mathbf{K}_T)) = \mathbf{J}_T^{-1} \text{Var}(N(\mathbf{0}, \mathbf{K}_T)) (\mathbf{J}_T^{-1})^{\top} = \mathbf{J}_T^{-1} \mathbf{K}_T \mathbf{J}_T^{-1}.
    $$ This yields the final expression for the approximate covariance matrix:
    \[
    \boldsymbol{\Sigma}_T = \mathbf{J}_T^{-1} \mathbf{K}_T \mathbf{J}_T^{-1}
    \]
    This derivation provides the explicit form for the approximate covariance matrix, justifying its use for statistical inference in large samples, as validated by the numerical experiments in Section \ref{sec:asymptotic_normality}.
\end{enumerate}




\bibliographystyle{elsarticle-num-names} 
\bibliography{references}
\end{document}